\documentclass{article}

\usepackage[british]{babel}
\usepackage[utf8]{inputenc}
\usepackage{amsmath}
\usepackage{amssymb}
\usepackage{amsthm}
\usepackage{verbatim}
\numberwithin{equation}{section}
\usepackage{mathtools}
\usepackage{graphicx}
\graphicspath{ {./images/} }

\newcommand{\fdd}{\overset{\textrm{f.d.d.}}\Rightarrow}
\newcommand{\R}{\mathbb R}

\newcommand{\abs}[1]{\left\vert#1\right\vert} 
\newcommand{\ind}{1\mkern-7mu1}

\newcommand\norm[1]{\left\lVert#1\right\rVert}
\DeclarePairedDelimiter\ceil{\lceil}{\rceil}

\usepackage[toc,page]{appendix}

\usepackage{hyperref}
\usepackage[framemethod=tikz]{mdframed}
\usepackage{amsthm}
\theoremstyle{definition}
\newtheorem{defi}{Definition}[section]
\newtheorem{ex}[defi]{Example}
\newtheorem*{asum*}{Assumption}

\newtheorem{rem}[defi]{Remark}
\theoremstyle{plain}
\newtheorem{lem}[defi]{Lemma}

\newtheorem{theorem}[defi]{Theorem}

\newtheorem*{theorem*}{Theorem}
\begin{document}
\title{Limit theorems for integrated trawl processes with symmetric L\'{e}vy bases}
\author{Anna Talarczyk\thanks{Institute of Mathematics, University of Warsaw, ul. Banacha 2, 02-097 Warsaw, Poland,  \hbox{e-mail:}
annatal@mimuw.edu.pl. Research supported in part by  National Science Center, Poland, grant  2016/23/B/ST1/00492.} \ 
and \L ukasz Treszczotko\thanks{Institute of Mathematics, University of Warsaw, ul. Banacha 2, 02-097 Warsaw, Poland,  \hbox{e-mail:}
l.treszczotko@mimuw.edu.pl. Research supported in part by  National Science Center, Poland, grant  2017/25/N/ST1/00368.}}
% \author{{\L ukasz Treszczotko\footnote{Institute of Mathematics, University of Warsaw, Banacha 2, 02-097 Warsaw}}\\
% %\href{mailto:lukasz.treszczotko@gmail.com}{lukasz.treszczotko@gmail.com} }
\date{June 30th,  2019}
\maketitle
\begin{abstract}
We study  long time behavior of integrated trawl processes introduced by Barndorff-Nielsen. The trawl processes form a class of stationary infinitely divisible processes, described by an infinitely divisible random measure (L\'evy base) and a family of shifts of a fixed set (trawl). We assume that the L\'evy base is symmetric and homogeneous and that the trawl set is determined by the trawl function that decays slowly. Depending on the geometry of the trawl set and on the L\'evy measure corresponding to the L\'evy base we obtain various types of limits in law of the normalized integrated trawl processes for large times. The limit processes are always stable and self-similar with stationary increments. In some cases they have independent increments -- they are stable L\'evy processes where the index of stability depends on the parameters of the model. We show that  stable limits with stability index smaller than $2$ may appear even in cases when the underlying L\'evy base has all its moments finite. In other cases, the limit process has dependent increments and it may be considered as a new extension of fractional Brownian motion to the class of stable processes.
\end{abstract}

\medskip

\\

\textup{2010} \textit{Mathematics Subject Classification}: \textup{Primary: 60G51, 60F17} \textup{Secondary: 60F05, 60G52, 60G18, 60G57}

\section{Introduction}\label{SEC_INTRO}

In this paper we investigate a class of stationary infinitely divisible processes. They have been introduced by Barndoff-Nielsen in \cite{BN2011} and studied further in \cite{BN2014} and \cite{GRAHOVAC}. Their discrete time counterparts were investigated in \cite{DOUKHAN}. Trawl processes are defined in the following way: suppose that $\Lambda$ is a homogenous  \emph{L\'{e}vy basis} on $\mathbb{R}^2$, that is, an infinitely divisible independently scattered random measure on $\mathbb{R}^2$, and let $A$ be a Borel subset of $\mathbb{R}^2$ with finite Lebesgue measure. Let $A_t$ denote $A$ shifted by the vector $(0,t)$, $A_t:= A + (t,0)$. A \emph{trawl process} is the process of the form 
\begin{equation}\label{X_def}
X_t = \Lambda(A_t), \quad t \in \mathbb{R}.
\end{equation}
The set $A$ is called a trawl. Since $\Lambda$ is homogeneous and infinitely divisible, the process $(X_t)_{t\geq 0}$ is stationary and infinitely divisible. To any L\'{e}vy basis there corresponds a L\'{e}vy process $L=(L(t))_{t\geq 0}$ (a process with stationary and independent increments). $L$ can be taken e.g. as $L(t):=\Lambda([0,t]\times [0,1])$. The process $L$ is called L\'{e}vy seed. The one-dimensional distributions of $X_t$ are determined by the choice of the L\'{e}vy seed and the dependence structure of a trawl process depends on the shape of the set $A$.

The processes of the form \eqref{X_def} are interesting mainly because they form a large class of processes that allows to model independently of each other the marginal distributions and the dependence structure. 

Typically,  the set $A$ is determined by a trawl function $g:[0,\infty) \rightarrow [0,\infty)$ with $\int_0^\infty g(s)ds <\infty$. More precisely we define
\[ A:=\{(x,y): x\leq 0, y \leq g(-x)\}\]
and then
\begin{equation}\label{A_t_def}
A_t =A+(t,0)= \{(x, y): x\leq t, 0 \leq y \leq g(t-x)\}.
\end{equation}

It seems quite clear that if $g$ vanishes sufficiently quickly, then the increments of $Y_T$ become asymptotically independent. 

A more interesting situation is when $g$ decays slowly, which will be the object of the current study. 
We will assume that the function $g$ is strictly decreasing, integrable and has a continuous derivative, that for large $t$ behaves as $const \times t^{-2-\gamma}$, for some $0<\gamma<1$. Typically one can think of $g$ of the form $C(1+t)^{-1-\gamma}$. It is known (see \cite{GRAHOVAC}) that if $g$ is regularly varying at infinity with index $-1-\gamma$, with $\gamma\in (0,1)$, then the corresponding trawl process is long range dependent. 

In the present paper we will investigate the behaviour of the integrated trawl process. More precisely, we study the convergence in law of the rescaled integrated trawl process
\begin{equation}\label{Y_def_primal}
Y_T(t)=\frac{1}{F_T}\int_0^{Tt}X_sds,
\end{equation}
as $T\rightarrow \infty$, where $F_T$ is an appropriate norming, chosen so that there exists a non-trivial limit in law.

Depending on the interplay between the type of decay of $g$ and the underlying L\'evy measure of the L\'evy base $\Lambda$ we show that the limit in law of \eqref{Y_def_primal} can be either a continuous stable process with dependent increments or a stable L\'evy process with index of stability depending on the parameters of the model.
 
\subsection{Background}\label{SEC_BACK}
Let us briefly describe the history of this problem and related results. \cite{GRAHOVAC} studied the behaviour of the integrated trawl process 
\begin{equation}\label{Y_non_scaled}
Y(t)=\int_0^t X_sds,
\end{equation}
with assumption on the trawl functions similar to ours, and in the case when the underlying L\'evy seed process has exponential moments.

It was shown that if one defines 
\begin{equation}
\tau_\ast(q):= \lim_{t\rightarrow \infty}\frac{\log \big(\mathbb{E}|Y(t)|^q\big)}{\log t},
\end{equation}
then there exists $q\geq 0$ such that for any $q^\ast \leq q$ one has $\tau_*(q)=q-\gamma$. This implies that for $q^\ast \leq p \leq q$
\[ \frac{\tau_*(p)}{p}\leq \frac{\tau_*(q)}{q}.\]
This property is known as \emph{intermittency}. In particular, intermittency implies that if the process $Y_T$ given by \eqref{Y_def_primal}
converges in the sense of finite dimensional distributions as $T\to \infty$  to some process $(Z_t)_{t\geq 0}$, then it is impossible to have convergence of all moments 
\[ \lim_{T\rightarrow \infty} \mathbb{E}\Big|\frac{Y_T(t)}{F_T}\Big|^q = \mathbb{E}|Z(t)|^q \]
for all $q>q^*$ and $t>0$. This follows form the fact that $Z$ would have to be self-similar with index $H$, i.e., $(Z(ct))_{t\geq 0} \overset{d}{=}c^H (Z(t))_{\geq 0}$ for all $c>0$, and $F_T$ of the form $F_T=T^HL(T)$ for some $H>0$ and a function $L$ which is slowly varying at $+\infty$, hence $\frac{\tau(q)}{q}$ would have to be constant. 
A natural question for us was to try to identify the limit process.
Indeed, as we shall see later, this corresponds to the situation of our Theorem \ref{THM2}, where the limit process of \eqref{Y_def_primal} is a stable process, with the stability parameter depending on the type of decay of the trawl function, even though $X_t$ has all moments finite.

\medskip

Another related paper is \cite{DOUKHAN}, where
discrete time trawl processes have been considered. They are of the form 
\[ X_k = \sum_{j=0}^\infty \gamma_{k-j}(a_j) \quad k \in \mathbb{Z},\]
where $\gamma_k = (\gamma_k(u))_{u\in \mathbb{R}}$ are i.i.d. copies of some process $\gamma = (\gamma(u))_{u\in \mathbb{R}}$ with $\gamma(u) \rightarrow 0$ in probability as $u\rightarrow 0$, and $a_j \in \mathbb{R}$, $j\in \mathbb{N}$, $\lim_{j\rightarrow \infty}a_j = 0$. $(X_k)$ is the trawl process corresponding to the seed process $\gamma$. In \cite{DOUKHAN} the behaviour of the process of partial sums
\[ S_n(t)=\frac{1}{F_n} \sum_{k=1}^{\ceil*{nt}} \Big(X(k) - \mathbb{E}X(k)\Big)\]
was investigated as $n \rightarrow \infty$ with an appropriate norming $F_n$. The authors considered the seed process with finite variance. Depending on the behaviour of the seed process and the trawl function $(a_j)$ various limits are obtained, either Gaussian limits: fractional Brownian motion and Brownian motion, or stable limits: $\alpha$-stable L\'{e}vy process. In particular, long memory trawl function $a_j \sim j^{-\alpha}$, $\alpha \in (1,2)$ and the standard Poisson seed process $\gamma$ leads to $\alpha$-stable L\'{e}vy process, even though with different norming the covariances converge to those of a fractional Brownian motion.

\subsection{Description of the results}\label{SEC_RES_DESC}
In this section we briefly describe our results. For precise statements of our theorems in their general form see Section \ref{SEC_RES}. We study the behaviour of the rescaled integrated trawl process $Y_T$ given by \eqref{Y_def_primal}. Our basic assumption is that $A_t$ is of the form \eqref{A_t_def} with the trawl function $g: [0,\infty) \rightarrow [0,\infty)$ which is integrable, strictly decreasing, has a continuous first derivative such that for large $t$ we have $g'(t)\sim -const \times t^{-2-\gamma}$ for some $0<\gamma<1$ (this corresponds to the assumption in \cite{DOUKHAN} that $a_j\sim j^{-1-\gamma}$ and to the assumptions made in \cite{GRAHOVAC}). In the latter paper the assumptions on $g$ were slightly less restrictive - $g'$ regularly varying at $+\infty$, but no limit in law theorems were established. 

We consider a homogeneous L\'{e}vy base $\Lambda$ such that for every $A \in \mathcal{B}(\mathbb{R}^2)$ (a Borel subset of $\R^2$) with finite Lebesgue measure, $\Lambda(A)$ is symmetric and does not have a Gaussian component, that is 
\begin{equation}\label{Lambda_char}
\mathbb{E}\exp(i \theta \Lambda(A)) =\exp\{-|A|\psi(\theta)\},
\end{equation}
where and $|A|$ is the Lebesgue measure of $A$, $\psi$ is the L\'evy exponent
\begin{equation}
\psi(\theta)=\int_\mathbb{R}\big(e^{i\theta y} - 1 -i\theta u\ind_{\{|y|<1\} }\big)\nu(dy),
\label{e:psi}
\end{equation}
and $\nu$ is a L\'{e}vy measure, i.e., a Borel measure on $\mathbb{R}$ satisfying
\begin{equation} \int_\mathbb{R} 1 \wedge |y|^2 \nu(dy)<\infty,
\label{e:nu}
\end{equation}
with $\nu(\{0\})=0$. We assume that $\nu$ is symmetric, hence \eqref{e:psi} can be written as
% \begin{equation}\label{Lambda_char_sym}
% % \mathbb{E}\exp(i \theta \Lambda(A)) = \exp\Big(|A|\int_\mathbb{R}\big(\cos(\theta y) - 1\big)\nu(dy)\Big).
% \mathbb{E}\exp(i \theta \Lambda(A)) = \exp\Big(|A|\int_\mathbb{R}\big(\cos(\theta y) - 1\big)\nu(dy)\Big).
% \end{equation}
% The corresponding L\'{e}vy exponent is 
\begin{equation}\label{psi_def_intro}
\psi(\theta) = \int_\mathbb{R}(1-\cos(\theta y))\nu(dy), \quad \theta \in \mathbb{R}.
\end{equation}
The assumption of symmetry simplifies some parts of the proofs, but we expect that it is not essential and it should be possible to obtain analogous results in the non symmetric case.

Depending on the behaviour of the L\'{e}vy measure $\nu$, or equivalently, on the behaviour of the L\'{e}vy exponent $\psi$, we obtain several types of limits for $Y_T$. All the limits are of course self-similar with stationary increments. We observe a phase transition - depending on the parameters of the model, the limit process may be an $\alpha$-stable process with dependent increments ($\alpha$ depends on $\nu$) or a stable L\'{e}vy process with index of stability which may be either $1+\gamma$ or smaller, depending on $\nu$.

For example, consider the case when $\Lambda$ is the standard independently scattered symmetric $\alpha$ stable random measure with Lebesgue control measure (i.e, $\nu(dx)=\frac{const}{|x|^{1+\alpha}}dx$ and $\psi(x)=\abs{x}^\alpha$, $0<\alpha<2$). 
\begin{itemize}
 \item 
If $\alpha>1+\gamma$ then $F_T = T^{1-\gamma/\alpha}$ and for any $\tau>0$ the process $Y_T$ converges in law in $\mathcal{C}[0,\tau]$ to an $\alpha$-stable process with dependent increments, which is of the form constant times the process
\begin{equation}\label{new_proc_def}
Y(t) = \int_0^\infty  \int_0^\infty \big(r_+\wedge t - (r-u)_+\wedge t \big) u^{-\frac{2+\gamma}{\alpha}}M_\alpha(drdu),
\end{equation}
where $M_\alpha$ is a symmetric $\alpha$-stable random measure on $\mathbb{R}_+^2$ with Lebesgue control measure. The integral is understood in the sense of \cite{ST}. The process $Y$ is self-similar with self-similarity index $H=1-\frac \gamma \alpha$, it has stationary increments and it is $\alpha$-stable, hence it may be thought of as yet another extension of fractional Brownian motion. 
\item
If $0<\alpha<1+\gamma$, then with the norming $F_T=T^{1/\alpha}$ we have 
\begin{equation}\label{RES_alpha_stable}
Y_T \fdd K Z_\alpha,
\end{equation}
where $Z_\alpha$ is a symmetric $\alpha$-stable L\'{e}cy process and $K$ is some finite constant. ($\fdd$ stands for convergence of finite dimensional distributions.)
\item
In the critical case $\alpha=1+\gamma$ we also have convergence \eqref{RES_alpha_stable} but the larger norming $F_T=T^{\frac{1}{\alpha}}\log T$. The appearance of the logarithm term is typical for the critical cases in many models. 
\end{itemize}

Another simple example covered by our techniques is the following:
\begin{itemize}
 \item 
Suppose that  $\nu$ is a finite measure such that
\[ \int_\mathbb{R}|x|^\kappa\nu(dx)<\infty\]
for some $\kappa>1+\gamma$. For example, $\Lambda$ can be a difference of two homogeneous Poisson random measures on $\R^2$. In this case the norming is $F_T=T^{\frac{1}{1+\gamma}}$ and the limit process is an $(1+\gamma)$-stable L\'{e}vy process. Note that the latter result corresponds to the one obtained in \cite{DOUKHAN} in the discrete time setting.
\end{itemize}
In the next section we formulate our results in their general form. Depending on the interplay of the L\'evy measure $\nu$ and the trawl function $g$, in the limit we obtain either the process $Y$ given by \eqref{new_proc_def} or stable L\'evy processes.

\medskip
The paper is organised as follows: in Section 2 we recall some of the basic notions and we state the results. Section 3 contains the proofs. There we start with the general scheme, later applying it to prove our theorems.

\medskip
\textbf{Notation.} By $C, C_1, C_2,\ldots$ we denote generic positive constants, whose value is not important to us. These constants may be different in different formulas. To help the reader we often write $C_1, C_2,\ldots$ to indicate that the constant changes from line to line.\\
$\fdd$ denotes convergence of finite dimensional distributions.\\
$C([0,\tau])$ with $\tau>0$ stands for the space of continuous functions from $[0,\tau]$ to $\R$.

\section{Results}\label{SEC_RES}

We assume that $\nu$ is a symmetric L\'{e}vy measure on $\mathbb{R}$. That is, $\nu$ is symmetric and satisfies \eqref{e:nu}.
% \begin{equation}
% \int_\mathbb{R}\big(1\wedge |x|^2\big)\nu(dx)<\infty.
% \end{equation}
We consider a homogeneous L\'evy basis $\Lambda$ on $\mathbb{R}^2$ corresponding to $\nu$, that is a family $\big(\Lambda(A)\big)_{A\in \mathcal{E}}$ of real-valued random variables where $\mathcal{E}$ denotes the class of Borel subsets of $\mathbb{R}^2$ with finite Lebesgue measure. $\Lambda$ satisfies the following conditions:
\begin{enumerate}
\item
$\Lambda$ is an independently scattered random measure, i.e., for any $A_1,A_2,\ldots \in \mathcal{E}$ with $A_j \cap A_i = \emptyset$ if $i  \neq j$, $\Lambda(A_1), \ldots \Lambda(A_k)$ are independent and if additionally $\bigcup_{j=1}^\infty A_j \in \mathcal{E}$, then 
\[\Lambda\big(\bigcup_{j=1}^\infty A_j\big) = \sum_{j=1}^\infty \Lambda(A_j).\]
\item
For any $A \in \mathcal{E}$
\begin{equation}
\mathbb{E}\exp\big(i\theta\Lambda(A)) = \exp(-|A|\psi(\theta)\big), \quad \theta \in \mathbb{R},
\end{equation}
where $|A|$ denotes the Lebesgue measure of $A$ and $\psi$ is the L\'evy exponent corresponding to $\nu$:
\begin{equation}\label{psi_def}
\psi(\theta)=\int_\mathbb{R}(1-\cos(\theta u))\nu(du).
\end{equation}
$\psi$ has this simple form because we have assumed the symmetry of $\nu$. Also, in our setting there is no drift or diffusion part.
\end{enumerate}

Integrals of deterministic functions with respect to general L\'evy bases were defined and studied in \cite{RR1989}. In our simple case, if a measurable function $f: \mathbb{R}^2 \rightarrow \mathbb{R}$ satisfies
\begin{equation}\label{RES_loc_0}
\int_{\mathbb{R}^2} \int_\mathbb{R} \big(uf(x))^2\wedge 1\big)\nu(du)dx <\infty,
\end{equation}
then the integral $I(f)=\int_{\mathbb{R}^2}f(x)\Lambda(dx)$ is well defined and 
\begin{equation}\label{I(f)_char}
\mathbb{E}\exp(i\theta I(f)) = \exp\Big(-\int_{\mathbb{R}^2}\int_\mathbb{R} \big(1-\cos(\theta u f(x))\big)\nu(du)dx\Big)
\end{equation}
(see [RR] and \cite{LRDSS} Appendix B.1.5). 

In particular, if $\Lambda$ is a symmetric $\alpha$-stable 
random measure, denoted by $M_\alpha$, that is corresponding to,  $\psi(x)=\abs{x}^\alpha$, and $\nu(dx)=\frac {C_\alpha}{\abs {x}^\alpha}$, then $I(f)=\int_{\R^2} fdM_{\alpha}$ is the integral considered in \cite{ST}. In this case $I(f)$ is well defined if
\begin{equation}
 \label{e:2.5a}
 \int_{\R^2} \abs{f(x)}^\alpha dx<\infty
\end{equation}
and
\begin{equation}\label{e:2.5b}
\mathbb{E}\exp\Big(i\theta \int_{\R^2}fdM_\alpha\Big) = \exp\Big(-\int_{\mathbb{R}^2}
\abs{f(x)}^\alpha dx \Big).
\end{equation}

We consider the trawl process described in the introduction. Suppose that $g: [0,\infty)\rightarrow [0,\infty)$ is a continuous, integrable, strictly decreasing function. We define
\[ A:=\{(x,y): x\leq 0, y \leq g(-x)\},\]
\[A_t = \{(x, y): x\leq t, 0 \leq y \leq g(t-x)\}\]
and set 
\[X_t = \Lambda(A_t), \quad t \geq 0.\]
For $T\geq 1$ we put
\begin{equation}\label{Y_T_defi_again}
Y_T(t) = \frac{1}{F_T}\int_0^{Tt}X_sds, \quad t\geq 0,
\end{equation}
where $F_T$ is an appropriate norming, which will be specified later. Our basic assumption on the trawl function $g$ is the following.
\begin{asum*}[\textbf{G}]
Assume that the trawl function $g$ is continuous, integrable, strictly decreasing, continuously differentiable on $(0, \infty)$ and its derivative satisfies 
\begin{equation}\label{ASSUM_G_deriv}
\lim_{x\rightarrow \infty} x^{2+\gamma}|g'(x)| =  C_g,
\end{equation}
for some $\gamma \in (0,1)$ and $C_g>0$. 
\end{asum*}

\begin{ex}
The function $g(x)=\frac{C}{(1+x)^{1+\gamma}}$ satisfies Assumption (\textbf{G}). For this function the proofs can be somewhat simplified since $g$ satsifies additionally 
\[ \sup_{x>0} x^{2+\gamma}|g'(x)|\leq C_1.\]
\end{ex}

% \begin{asum*}[\textbf{A}]
% Assume that there exists $\alpha \in (0,2)$ and $C_\psi>0$ such that
% \begin{equation}\label{ASSUM_A}
% \lim_{T\rightarrow \infty} \frac{\psi(x)}{|x|^\alpha} = C_\psi.
% \end{equation}
% \end{asum*}

Now we are ready to state our main results.

\begin{theorem}\label{THM1}
Suppose that assumption (\textbf{G}) is satisfied and  that there exists $\alpha \in (0,2)$ and $C_\psi>0$ such that
\begin{equation}\label{ASSUM_A}
\lim_{\abs{x}\rightarrow \infty} \frac{\psi(x)}{|x|^\alpha} = C_\psi.
\end{equation}
Moreover, assume that  $\alpha>1+\gamma$ and there exists some $\kappa>1+\gamma$ such that
\begin{equation}\label{thm1_cond}
\int_{|y|\geq 1} |y|^\kappa \nu(dy) <\infty.
\end{equation}
Let $Y_T$ be given by \eqref{Y_T_defi_again} with
\begin{equation}\label{thm1_norming}
F_T = T^{\frac{\alpha -\gamma}{\alpha}}.
\end{equation}
Then, for any $\tau > 0$, the processes $Y_T$ converge in law in $\mathcal{C}([0,\tau])$, as $T\rightarrow \infty$, to the process $KY$, where $Y$ is defined by \eqref{new_proc_def} and $K$ is a positive constant.
\end{theorem}

% \begin{ex}
% If $\nu(dx)=\frac{c_\alpha}{|x|^{1+\alpha}}dx$ for some $\alpha \in (0,2)$ and \eqref{ASSUM_A} is satisfied, then $\Lambda$ is simply a homogeneous symmetric $\alpha$-stable random measure on $\mathbb{R}^2$ and the integrals with respect to $\Lambda$ correspond to those discussed for example in Chapter 3 of \cite{SPLRD}.
% \end{ex} 

\begin{rem}\label{rem_1_3}
Whether or not condition \eqref{ASSUM_A} holds depends only on the behaviour of the L\'evy measure $\nu$ near $0$ since for any $\epsilon>0$ the function 
\[ u \mapsto \int_{|x|>\epsilon}(1-\cos(ux))\nu(dx)\]
is bounded. If near zero $\nu$ has a density $h(x)$ such that 
\begin{equation}\label{zero_density_1}
\lim_{x\rightarrow 0}|x|^{1+\alpha}h(x) = C
\end{equation}
for some finite positive $C$ and $\alpha \in (0,2)$, then \eqref{ASSUM_A} is satisfied.
\end{rem}

\begin{rem}
For $\alpha=2$ (i.e. when $\Lambda$ is a homogeneous Gaussian random measure) one can prove a result similar to the one of  Theorem \ref{THM1}. In this case the limit process  turns out to be fractional Brownian motion with Hurst coefficient $1-\gamma/2$. Therefore, we may think of our limit process $Y$ as a yet another extension of fractional Brownian motion to the realm of stable processes.
\end{rem}

\begin{rem} We have written a basic code to simulate the process $Y$. We include a picture of sample paths obtained. The interested reader may look up the Python code on the GitHub repository \footnote{\url{https://github.com/lukasz-treszczotko/trawl_processes_limits}}.
\end{rem}

\begin{figure}
\textbf{Simulated sample path trajectories for various pairs of $\alpha$ and $\gamma$}\par\medskip
\begin{tabular}{p{0.5\textwidth} p{0.5\textwidth}}
  \vspace{0pt} \includegraphics[scale=0.3]{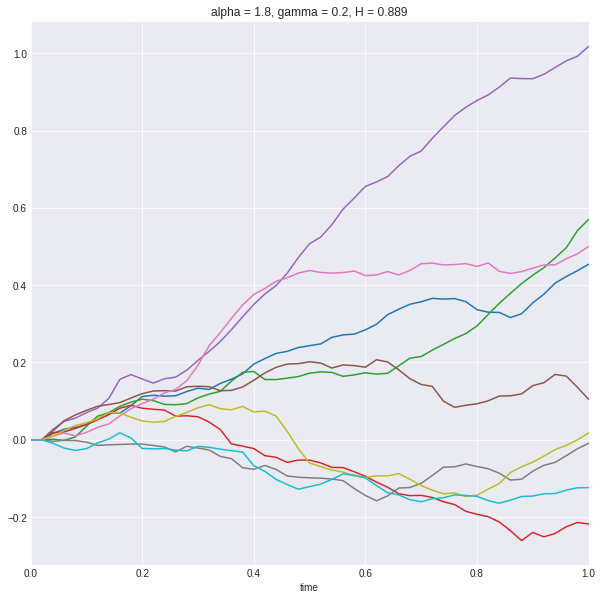} 
  \vspace{0pt} \includegraphics[scale=0.3]{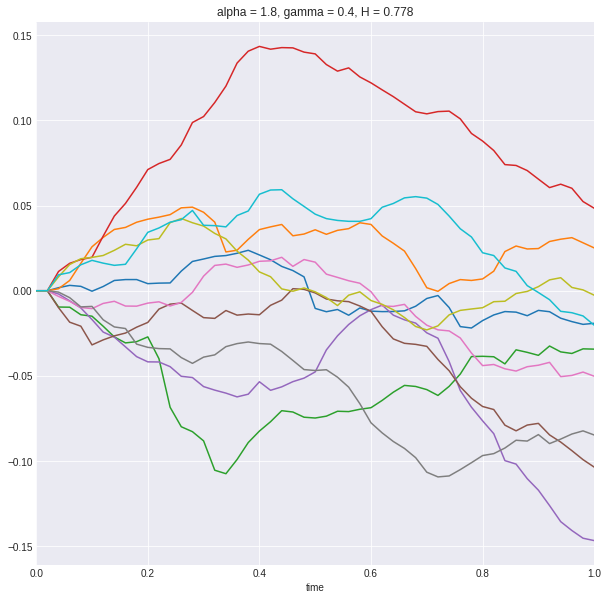} &
  \vspace{0pt} \includegraphics[scale=0.3]{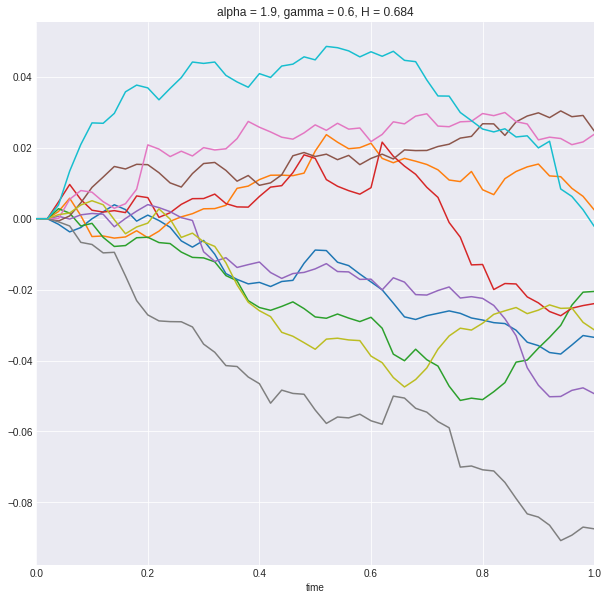} 
  \vspace{0pt} \includegraphics[scale=0.3]{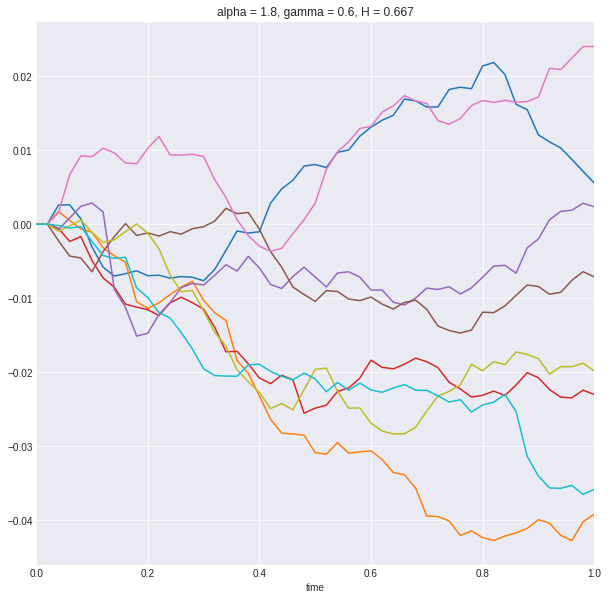} 
\end{tabular}
\end{figure}

\begin{theorem}\label{THM2}
Assume (\textbf{G}) and either
\begin{itemize}
\item[(i)]
\begin{equation}\label{thm2_cond1}
\psi(u)\leq C|u|^\kappa \wedge|u|^\alpha, \quad u \in \mathbb{R}
\end{equation}
for some $2 \geq \kappa > 1+\gamma$, $0 \leq \alpha <1+\gamma$ and finite constant $C>0$, or
\item[(ii)] suppose that $\psi$ is nondecreasing on $[0,\infty)$,
\begin{equation}\label{thm2_cond2a}
\int_\mathbb{R}\psi(u)|u|^{-2-\gamma}du <\infty,
\end{equation}
and 
\begin{equation}\label{thm2_cond2b}
\sup_{u \geq 0}u^{2+\gamma}|g'(u)| \leq C
\end{equation}
for some finite constant $C>0$. 
\end{itemize}
Set 
\begin{equation}\label{thm2_norming}
F_T = T^{\frac{1}{1+\gamma}}.
\end{equation}
Then for $Y_T$ given by \eqref{Y_T_defi_again} we have
\begin{equation*}
 Y_T\fdd K\xi^{(1+\gamma)}, \qquad \text{as}\ T\to \infty,
\end{equation*}
where $\xi^{(1+\gamma)}$ denotes a symmetric $(1+\gamma)$-stable L\'evy process and $K$ is a positive constant.
% $Y_T$ given by \eqref{Y_T_defi_again} converges in the sense of finite-dimensional distributions to the process $(K\eta_t)_{t\geq 0}$ where $K$ is some finite positive constant and $\eta$ is a symmetric $1+\gamma$-stable L\'{e}vy process.
\end{theorem}

\begin{rem}\label{thm2_REM1}
\begin{itemize}
\item[(a)]
Note that \eqref{thm2_cond1} implies \eqref{thm2_cond2a}. Condition \eqref{thm2_cond2a} is slightly weaker, but in order to prove convergence under this assumption, we need to assume something more about the trawl function $g$.
\item[(b)]
If 
\begin{equation}\label{thm2_rem_1}
\int_{\{|x|<1\}}|x|^\alpha \nu(dx) < \infty
\end{equation}
and \eqref{thm1_cond} is satsified for $2 \geq \kappa > 1+\gamma > \alpha \geq 0$, then using $|1-\cos(x)| \leq 2 \wedge x^2 \leq 2|x|^\delta$ for any $0 \leq \delta \leq 2$ we obtain
\begin{eqnarray*}
\psi(u)& =& \int_{\{|x|<1\}} (1-\cos(ux))\nu(dx)+ \int_{\{|x|\geq 1\}} (1-\cos(ux))\nu(dx) \\
& \leq & 2|u|^\alpha \int_{\{|x|<1\}}|x|^\alpha \nu(dx) + 2 \nu(\{|x|\geq 1\}) \wedge \Big(|u|^\kappa \int_{\{|x|\geq 1\}}|x|^\kappa\nu(dx)\Big),
\end{eqnarray*}
hence \eqref{thm2_cond1} holds. In particular, if $\nu$ is a finite measure and satisfies \eqref{thm1_cond}, then \eqref{thm2_cond1} holds. Similarly as in Remark \ref{rem_1_3}, if near zero $\nu$ has density satisfying \eqref{zero_density_1} and \eqref{thm1_cond} holds , then \eqref{thm2_cond1} is satisfied.
\end{itemize}
\end{rem}

As a direct consequence of Theorem \ref{THM2} we obtain the following result.

\begin{ex}\label{col_2_9}
If $\Lambda = N^{(1)} - N^{(2)}$, where  $N^{(1)}$ and $N^{(2)}$ are two independent  Poisson random measures on $\R^2$ with Lebesgue intensity measure,  then the processes $Y_T$ converge in the sense of finite-dimensional distributions to a symmetric $(1+\gamma)$--stable L\'evy process multiplied by a constant. In this case $\nu = \lambda(\delta_1 + \delta_{-1})$ for some $\lambda>0$ and $\psi(x)=2\lambda(1-\cos(x))$. This result is a symmetrized continuous time analogue of the discrete time result of \cite{DOUKHAN}.
\end{ex}

\begin{theorem}\label{THM3}
Assume that (\textbf{G}) is satisfied and that there exist $0<\alpha < 1+\gamma$ and a finite constant $C_\alpha>0$ such that
\begin{equation}\lim_{x\rightarrow 0} \frac{\psi(x)}{|x|^\alpha} = C_\alpha.
 \label{e:2.16a}
\end{equation}
Furthermore, assume that there exist $C>0$ and $0\le {\kappa}<1+\gamma$ such that 
\begin{equation} \psi(u) \leq C(1 \vee |u|^{{\kappa}}), \quad u \in \mathbb{R}.
\label{e:2.18a}
\end{equation}
 Let $Y_T$ be defined by \eqref{Y_T_defi_again} with
\begin{equation}\label{thm3_norming}
F_T = T^{\frac{1}{\alpha}}.
\end{equation}
Then
\begin{equation*}
 Y_T\fdd K\xi^{(\alpha)}, \qquad \text{as}\ T\to \infty,
\end{equation*}
where $K$ is a positive constant and $\xi^{(\alpha)}$ is a symmetric $\alpha$-stable L\'evy process.
% 
% Then $Y_T$ given by \eqref{Y_T_defi_again} converges in the sense of finite-dimensional distributions to the process $(K\xi_t)_{t\geq 0}$, where $K$ is some finite positive constant and $\xi$ is a symmetric $\alpha$-stable L\'evy process.
\end{theorem}

Let us now see how these general theorems work in the case of symmetric $\alpha$-stable random measures.

\begin{ex}\label{alpha_stable_example}
Suppose that $\Lambda$ is a homogeneous and symmetric $\alpha$-stable random measure on $\mathbb{R}^2$  with $\alpha \in (0,2)$. We also assume that (\textbf{G}) is staisfied. This case corresponds to 
\[ \psi(x) = |x|^\alpha, \quad  x \in \mathbb{R}\]
and 
\[ \nu(dx) = \frac{c_\alpha}{|x|^{\alpha+1}}, \quad x \in \mathbb{R}.\]
\begin{itemize}
\item 
If $\alpha>1+\gamma$, then \eqref{thm1_cond} holds 
% \[ \int_{|y|\geq 1}|y|^\kappa \nu(dy) <\infty \]
for any $1+\gamma<\kappa <\alpha$, hence, the assumptions of Theorem \ref{THM1} are satisfied, and with the norming $F_T = T^{1-\gamma/\alpha}$, for any $\tau>0$, the process $Y_T$ converges in law in $\mathcal{C}([0,\tau])$ to the process $KY$, where $K$ is some finite constant and $Y$ is given by \eqref{new_proc_def}.
\item
If $\alpha<1+\gamma$, then the assumptions of Theorem \ref{THM3} are satisfied and with the normalization $F_T=T^{1/\alpha}$, the process $Y_T$ converges in the sense of finite-dimensional distributions to symmetric $\alpha$-stable L\'evy process multiplied by a constant.
\end{itemize}
\end{ex}

In the next theorem we will discuss the critical case $\alpha=1+\gamma$.

\begin{theorem}\label{THM_CRITICAL}
Assume that $\Lambda$ is a  symmetric $\alpha$-stable random measure on $\mathbb{R}^2$. Also, suppose that (\textbf{G}) is satisfied and $\alpha = 1+\gamma$. Let $Y_T$ be defined by \eqref{Y_T_defi_again} with
\[ F_T = T^{1/\alpha}\log(T).\] 
Then
\begin{equation*}
 Y_T\fdd K\xi^{(\alpha)}, \quad \text{as}\ T\to \infty,
\end{equation*}
% 
% Then $Y_T$ converges in the sense of finite-dimensional distributions to $(K\xi_t)_{t\geq 0}$,
where
$K$ is some finite positive constant and $\xi^{(\alpha)}$ is a symmetric $\alpha$-stable L\'evy process. 
\end{theorem}

Thus, in the case of $\alpha$-stable random measures we have a phase transition: for large $\alpha$ ($\alpha>1+\gamma$) the limit process has dependent increments, while for small $\alpha$ ($\alpha<1+\gamma$) the limit process has independent increments. In the critical case ($\alpha=1+\gamma$) the limit process also has independent increments but the norming differs by a logarithmic factor. This type of phase transition and existence of two regimes - one in which the limit process has independent increments and  another one in which the increments are  dependent, along with the logarithmic factor in the norming in the critical case is a typical behavior, also observed in other models. See for example 
\cite{BOJTALIV} and \cite{functlim4} for a model with behaviour of this type, related to occupation time processes of branching particle systems.

\section{Proofs}

\subsection{General scheme}

In all the proofs we show convergence of finite-dimensional distributions by proving convergence of the corresponding characteristic functions. In Theorem \ref{THM1} we additionally show tightness in $\mathcal{C}([0,\tau])$ for all $\tau>0$. 
We start with some general calculations used in all the cases. 

First we write 
 the process $Y_T$ in a different form, given by the lemma below.

\begin{lem}\label{lem:3.1a}
Let $Y_T$ be given by \eqref{Y_T_defi_again}. Then
\begin{equation}\label{Y_repr_eq}
Y_T(t) = \frac{1}{F_T} \int_{\mathbb{R}^2}\left( \left(x+g^{-1}(y)\right)_+ \wedge (Tt) - x_+ \wedge (Tt)\right)\ind_{\{0 \leq y \leq g(0)\}}\Lambda(dx,dy)
\end{equation}
\end{lem}

\begin{proof}
It is immediate to see that
\begin{eqnarray}
\int_0^{t}\ind_{A_s}(x, y)ds &=& \int_0^t\ind_{\{x \leq s \leq g^{-1}(y)+x\}}dx \, \ind_{\{0 \leq y \leq g(0)\}}\nonumber \\
&=&\Big(\big(g^{-1}(y)+x\big)_+\wedge t - x_+ \wedge t\Big)\ind_{\{0 \leq y \leq g(0)\}}.
\end{eqnarray}
Hence \eqref{Y_repr_eq} follows from the Fubini theorem for L\'evy bases (see Theorem 3.1 in \cite{BNConnor2011}). Note that this theorem can be applied directly in the case $\int_\mathbb{R}|y|\wedge |y|^2 \nu(dy)<\infty$. If we do not assume $\int_{\{|y|>1\}}|y| \nu(dy)<\infty$, then we can decompose 
\begin{equation}\label{Lambda_decomposition}
\Lambda = \Lambda_1 + \Lambda_2,
\end{equation}
where $\Lambda_1$ and $\Lambda_2$ are independent L\'evy bases corresponding to L\'evy measures $\nu_1$ and $\nu_2$, respectively, where 
\begin{eqnarray}\label{nu_decomposition}
\nu_1(B) & = & \nu(B\cap \{x: |x|<1\}), \\
\nu_2(B) & = & \nu(B\cap \{x: |x|\geq 1\}) 
\end{eqnarray}
for $B$ a Borel set in $\mathbb{R}$. Then $\Lambda_1$ satisfies the assumptions of Theorem 3.1 in \cite{BNConnor2011} and $\Lambda_2$ can be written as
\[\Lambda_2 = \sum_i \eta_i \delta_{(x_i, y_i)},\]
where $(x_i, y_i)$ are points of a Poisson random measure on $\mathbb{R}^2$ with 
Lebesgue intensity measure, multiplied by $\nu_2(\mathbb{R}^2)$
% intensity measure $\otimes \lambda_2$, where $\lambda_2$ denotes two-dimensional Lebesgue measure on $\mathbb{R}^2$
and $\eta_i$ are i.i.d. random variables with law $\nu_2(\cdot)/\nu_2(\mathbb{R}^2)$, independent of the Poisson random measure. The trawl function is non-decreasing and integrable, thus
\[\sup_{0 \leq s \leq Tt}\ind_{A_s}(x,y) \leq \ind_{A_0 \cup [0, Tt]\times [0,g(0)] }(x,y).\]
Only a finite number of points $(x_i, y_i)$ of the Poisson random measure belong to 
$A_0 \cup [0,g(0)] \times [0, Tt]$, hence we can exchange the order of integration with respect to $ds$ and $\Lambda_2$  as well and \eqref{Y_repr_eq} follows.  
\end{proof}

Note that in some of the proofs it will be convenient to use the decomposition \eqref{Lambda_decomposition}
of $\Lambda$. Then
\begin{equation}
 \label{e:3.5a}
 Y_T=Y_{T,1}+Y_{T,2}, 
\end{equation}
where 
 $Y_{T,1}$ and $Y_{T,2}$, are independent processes of the form \eqref{Y_T_defi_again}, corresponding to $\Lambda_1$ and $\Lambda_2$, 
 respectively. We also denote the corresponding characteristic exponents by
\begin{align}
\psi_1(\theta) =& \int_{\{|x|<1\}}(1-\cos(\theta x))\nu(dx),\quad \theta\in \R \label{e:3.5b}\\
\psi_2(\theta) =& \int_{\{|x|\geq 1\}}(1-\cos(\theta x))\nu(dx), \quad \theta \in \mathbb{R}.\label{psi_decomposition}
\end{align}

As the next step we write the characteristic function of $Y_T$. We need some additional notation.  Denote 
\begin{equation}\label{f_def}
f(t,r,u):= r_+ \wedge t - (r-u)_+ \wedge t = \int_0^t \ind_{[r-u,r]}(s)ds \quad t,r,u \geq 0.
\end{equation}

We have the following lemma describing the characteristic function of finite-dimensional distributions of $Y_T$.

\begin{lem}\label{Y_T_char_lemma}
Fix $T>0$, $a_1,\ldots,a_n \in \mathbb{R}$, $0\leq t_1\le \ldots \le t_k<+\infty$ and denote
\begin{equation}\label{h_T_def}
h_T(r, u) = \sum_{j=1}^n a_j f(Tt_j,r,u), \qquad  r,u\ge 0 .
\end{equation}  
Then, 
For $Y_T$ defined by \eqref{Y_T_defi_again} we have
\begin{equation}\label{Y_T_char_func}
\mathbb{E}\exp \bigg(i\sum_{j=1}^k a_j Y_T(t_j)\bigg) = \exp\bigg(-\int_{\mathbb{R}_+^2}\psi\Big(\frac{1}{F_T}h_T(r,u)\Big)\abs{g'(u)}drdu\bigg),
\end{equation}
where $\psi$ is the L\'evy exponent \eqref{psi_def}.
% 
% and $\psi$ is the characteristic exponent of $\Lambda$ given by \eqref{psi_def}.
\end{lem}
\proof
By Lemma \ref{lem:3.1a}, \eqref{I(f)_char} and \eqref{psi_def} we have
\begin{multline*}
 \mathbb{E}\exp \bigg(i\sum_{j=1}^k a_j Y_T(t_j)\bigg) 
\\= \exp\left\{ -\int_{\R^2} \psi\left(\frac 1{F_T}\sum_{j=1}^n a_j \left[ (g^{-1}(y)+x)\wedge (Tt_j)-x_+\wedge (Tt_j)\right]\ind_{0<y<g(0)}\right)dxdy \right\}
\end{multline*}
Next we  substitute $u=g^{-1}(y)$ and $r=x+g^{-1}(y)$. We also observe that   if $r\le 0$ we have $\left(r_+\wedge (t_jT)-(r-u)_+\wedge (t_jT)\right))\ind_{\{u>0\}}=0$. Hence \eqref{Y_T_char_func} follows.\qed

% The proof follows immediately from Theorem 3.3.2 in \cite{SPLRD} and we skip it.

The formula \eqref{Y_T_char_func} will be our starting point of the proofs of convergence of finite dimensional distributions in 
 Theorems \ref{THM1}, \ref{THM2} and \ref{THM3}. We will show that the right-hand side of \eqref{Y_T_char_func} converges to 
\[ \mathbb{E}\exp\Big(i\sum_{j=1}^n a_j \widetilde{Y}(t_j)\Big),\]
where $\widetilde{Y}$ is the corresponding limit process. 

This will amount to proving convergence of the term in the exponent on the right hand side of \eqref{Y_T_char_func}, which we denote by $I(T)$.
\begin{equation}
 \label{e:I_T}
 I(T)=\int_{\R^2_+}\psi\left(\frac 1{F_T}h_T(r,u)\right)\abs{g'(u)}du.
\end{equation}

\begin{comment}

Recall that we define the rescaled integrated trawl process by 
\begin{equation}\label{Y_def}
Y_T(t):= \frac{1}{F_T}\int_0^{Tt}X_sds, \quad t\geq 0
\end{equation}
for all $T>0$. Notice that 
\begin{equation*}
Y_T(t) = \int_{\mathbb{R}^2}v_t^T(r,x)\Lambda(dr,dx),
\end{equation*}
where 
\begin{equation*}
v_t^T(r,x)=\frac{1}{F_T}\int_0^{Tt}\ind_{A_s}(r, x)ds.
\end{equation*}
Since
\begin{eqnarray}
\int_0^{t}\ind_{A_s}(r, x)ds &=& \int_0^t\ind_{\{r \leq s \leq g^{-1}(x)+r\}}dr \times \ind_{\{0 \leq x \leq g(0)\}}\nonumber \\
&=& \Big(\big(g^{-1}(x)+r\big)_+\wedge t = r_+ \wedge t\Big)\times\ind_{\{0 \leq x \leq g(0)\}} ,
\end{eqnarray}
we see that the following lemma holds.
\begin{lem}\label{Y_repr_lem}
The process $Y_T$ can be written as
\begin{equation}\label{Y_repr_eq}
Y_T(t) = \frac{1}{F_T} \int_{\mathbb{R}^2}\left( \left(r+g^{-1}(x)\right)_+ \wedge (Tt) - r_+ \wedge (Tt)\right)\ind_{\{0 \leq x \leq g(0)\}}\Lambda(dr,dx)
\end{equation}
\end{lem}

Notice that by Theorem 3.3.2 in \cite{SPLRD}, for any $a_1,\ldots,a_k \in \mathbb{R}$ and $0\leq t_1,\ldots, t_k<+\infty$ we have 
\begin{equation}\label{Y_T_char_func}
\mathbb{E}\exp \bigg(i\sum_{j=1}^k a_j Y_T(t_j)\bigg) = \exp\bigg(-\psi\Big(\frac{1}{F_T}h_T(r,u)\Big)drdu\bigg),
\end{equation}
where
\begin{equation}\label{h_T_def}
h_T(r, u) = \sum_{j=1}^n a_j \left( \left(r+g^{-1}(u)\right)_+ \wedge (Tt_j) - r_+ \wedge (Tt_j)\right)\ind_{\{0 \leq u \leq g(0)\}}
\end{equation}
and $\psi$ is the characteristic exponent of $\Lambda$.

\end{comment}

\subsection{Auxiliary estimates and identities}
\medskip
We will frequently use the following simple facts concerning $f$ and $h_T$

\begin{lem}\label{f_lem}
Let $f$ be given by \eqref{f_def} and $h_T$ as in Lemma \ref{lem:3.1a}.  Then
\begin{itemize}
\item[(i)]
\begin{alignat}{2}
  0 \leq  f(t,r,u) &\leq t \wedge u \wedge r &\qquad &r,u,t\ge 0,\label{f_estimates}\\
  f(t,r,u) &=0   &\qquad &\textrm {for}\ t\ge 0 \ \textrm{and }\ r>t+u,\label{e:3.11a}\\
|h_T(r,u)| & \leq  \left(\sum_{j=1}^n |a_j|\right) f(Tt_n, r, u), &\qquad &r,u \geq 0.\label{e:hT_estimate}
\end{alignat}

%\begin{comment}
%\begin{align}
%\label{f_estimates}
%0 \leq  f(t,r,u) &\leq t \wedge u \wedge r, \qquad  r,u,t\ge 0, %\\
%f(t,r,u) &=0, \quad \textrm {for}\ t\ge 0 \ \textrm{and }\ r>t%%%+u,\label{e:3.11a}\\
%|h_T(r,u)| & \leq \left(\sum_{j=1}^n |a_j|\right) f(Tt_n, r, u), %\quad r,u \geq 0.
%\label{e:hT_estimate}
%\end{align}
%\end{comment}

\item[(ii)]
If, additionally,
we assume that $\kappa>1+\gamma >1$ %  and let $f$ be given by \eqref{f_def}.
then there exists a constant $C>0$ depending only on $\kappa$ and $\gamma$, such that for all $t\geq 0$ we have
\begin{equation}
\int_0^\infty \int_0^\infty |f(t,r,u)|^\kappa u^{-2-\gamma}du dr = Ct^{\kappa-\gamma}.\label{e:3.15}
\end{equation}
\end{itemize}
\end{lem}

\begin{proof}
Part \textit{(i)} is a direct consequence of \eqref{f_def} and \eqref{h_T_def}.

To prove \textit{(ii)}
observe that by \eqref{f_estimates} and \eqref{e:3.11a} for $t=1$ we have 
\begin{eqnarray*}
\int_0^\infty \int_0^\infty\big|f(1,r,u)\big|^\kappa u^{-2-\gamma}drdu & = & \int_0^\infty \int_0^{1+u}\big|f(1,r,u)\big|^\kappa u^{-2-\gamma}drdu \\
& \leq & \int_0^1 \int_0^{1+u} u^\kappa u^{-2-\gamma}drdu \\
&& \: + \int_1^\infty \int_0^{1+u}u^{-2-\gamma}drdu <+\infty
\end{eqnarray*}
since $\kappa>1+\gamma$.
Now, using 
\[f(t,r,u) = tf(1, r/t, u/t),\]
\eqref{e:3.15} follows by a simple substitution.
% hence
% \[\int_0^\infty \int_0^\infty |f(t,r,u)|^\kappa u^{-2-\gamma}du dr = t^{\kappa-\gamma}\int_0^\infty \int_0^\infty\big|f(1,r,u)\big|^\kappa u^{-2-\gamma}drdu.\]
% Using the fact that
% \[f(t,r,u) \leq t \wedge u \wedge r \]
% and $f(t,r,u)=0$ for $r>t+u$, we have
% \begin{eqnarray*}
% \int_0^\infty \int_0^\infty\big|f(1,r,u)\big|^\kappa u^{-2-\gamma}drdu & = & \int_0^\infty \int_0^{1+u}\big|f(1,r,u)\big|^\kappa u^{-2-\gamma}drdu \\
% & \leq & \int_0^1 \int_0^{1+u} u^\kappa u^{-2-\gamma}drdu \\
% && \: + \int_1^\infty \int_0^{1+u}u^{-2-\gamma}drdu <+\infty
% \end{eqnarray*}
% since $\kappa>1+\gamma$.
\end{proof}

\subsection{Proof of Theorem \ref{THM1}}

%\begin{proof}

First observe that by part (ii) of  Lemma \ref{f_lem}, \eqref{e:2.5a} and \eqref{e:2.5b}
it follows that the   process $Y$ given by \eqref{new_proc_def} is well defined.

We will  show convergence of finite-dimensional distributions and then establish tightness on any interval $[0,\tau]$, $\tau>0$, which suffices to obtain the desired convergence (see Thm. 8.1 in \cite{BILL})

\medskip

\textbf{Step 1. Convergence of finite dimensional distributions}

Fix any $a_1,\ldots, a_n \in \mathbb{R}$ and $0\le t_1\le \ldots \le t_n $ and recall the notation \eqref{e:I_T} and \eqref{h_T_def}.
Let us also denote
\begin{equation}\label{h_def}
h(r, u) =\sum_{j=1}^n a_j  f(t_j,r,u)= \sum_{j=1}^n a_j  \left(r_+ \wedge t_j - (r-u)_+ \wedge t_j\right), \quad r,u\ge 0.
\end{equation}
% \begin{equation}
%  \label{e:I}
%  I(T)=\int_{\mathbb{R}_+^2}\psi\left(\frac{1}{F_T}h_T(r,u)\right)\abs{g'(u)}drdu,
% \end{equation}
% where  $h_T$ is given by \eqref{h_T_def}.
Using \eqref{Y_T_char_func}, \eqref{e:I_T} and \eqref{e:2.5b},
to prove convergence of finite-dimensional distributions, we  only have to show % (using \eqref{Y_T_char_func}) 
that
\begin{equation}\label{fdd_conv}
\lim_{T\rightarrow\infty}I(T) = K^\alpha \int_{\R_+^2} \left|h(r,x)\right|^\alpha u^{-2-\gamma}drdu,
\end{equation}
for some finite positive constant $K$. 

By \eqref{e:I_T}, \eqref{h_T_def}, \eqref{h_def} and recalling the definition of $F_T$ \eqref{thm1_norming} we have
\begin{align}
  I(T)=&\int_{\R_+^2} T^2 \psi\left(\frac{T}{F_T}h(r,u)\right)\abs{g'(Tu)}drdu\label{e:3.18a} \\
  =&\int_{\R_+^2}\left(\frac T{F_T} \right)^{-\alpha}\psi\left(\frac{T}{F_T}h(r,u)\right) T^{2+\gamma}\abs{g(Tu)}dr dt.\label{e:3.18b}
\end{align}
By \eqref{ASSUM_G_deriv} and \eqref{ASSUM_A} we see that the integrand converges pointwise to the integrand on the right hand side of \eqref{fdd_conv}. Therefore, to prove \eqref{fdd_conv} it remains to justify the passage to the limit under the integral.

We will use the decomposition \eqref{e:3.5a}, which corresponds to $\psi=\psi_1+\psi_2$, where $\psi_1$ and $\psi_2$ are given by \eqref{e:3.5b} and \eqref{psi_decomposition}, respectively. We write
\begin{equation}
 I(T)=I_1(T)+I_2(T),
 \label{e:3.19}
\end{equation}
where $I_1(T)$ and $I_2(T)$ are defined by \eqref{e:I_T} with $\psi$ replaced by $\psi_1$ and $\psi_2$, respectively.

We will show that
\begin{align}
 \lim_{T\to\infty}I_1(T)=& K^\alpha \int_{\R_+^2} \left|h(r,x)\right|^\alpha u^{-2-\gamma}drdu,\label{e:3.21}\\
 \lim_{T\to \infty}I_2(T)=&0.\label{e:3.22}
\end{align}
This will imply
\begin{equation}
 Y_{T,1}\fdd KY\qquad \textrm{and} \qquad Y_{T,2}\fdd 0.
\label{e:fdd_Y_T_i}
 \end{equation}
 As the limit of $Y_{T,2}$ is deterministic, $Y_{T,2}(t)$ converges to $0$ in probability for any $t>0$, hence \eqref{e:fdd_Y_T_i} implies the desired convergence of finite-dimensional distributions of $Y_T$.
 
Observe, that by the estimate $1-\cos(\theta x)\le (\theta x)^2$, \eqref{e:3.5b} and \eqref{ASSUM_A} we have
\begin{equation}
 0\le \psi_1(x)\le C(\abs{x}^\alpha\wedge \abs{x}^2)\le C\abs{x}^\alpha.
 \label{e:3.23}
\end{equation}
We may assume that $\kappa$ in the assumptions of the Theorem satisfies $1+\gamma<\kappa<\alpha$, since if \eqref{thm1_cond} holds for some $\kappa$, then it also holds for smaller $\kappa$. In particular, $\kappa<2$. Then, using $(1-\cos(x\theta))\le 2 \abs{\theta x}^\kappa$  \eqref{psi_decomposition} and \eqref{e:nu} we have
\begin{equation}
 \psi_2(x)\le C\left(\abs{x}^\kappa\wedge 1\right)\le C\abs{x}^\kappa.
\label{e:3.19d}
 \end{equation}
Since $\psi_2$ is bounded and $\alpha>1+\gamma>0$ we have
\begin{align}
 \lim_{\abs{x}\to \infty}\frac{\psi_1(x)}{\abs{x}^\alpha}=&\lim_{\abs{x}\to \infty}\frac{\psi(x)}{\abs x^\alpha}=C_\psi, \label{e:3.25a}\\
 \lim_{\abs{x}\to \infty}\frac{\psi_2(x)}{\abs{x}^\alpha}=&0.\label{e:3.25b}
\end{align}
Moreover, by Assumption (${\textbf{G}}$) there exists $D>0$ such that 
\begin{equation}
 \label{e:3.19a}
 \sup_{u\ge D}\abs{g'(u)}u^{2+\gamma}\le 2C_g,
\end{equation}
and we may therefore write
\begin{equation}
 I_i(T)=A_i(T)+ B_i(T), \qquad i=1,2,
\end{equation}
where
\begin{align}
A_i(T) =&\int_0^{D/T} \int_0^\infty \Big(\frac{T}{F_T}\Big)^{-\alpha}\psi_i \Big(\frac{T}{F_T}h(r,u)\Big) T^{2+\gamma}|g'(Tu)|dr  du\quad i=1,2
\label{1A2def}\\
B_i(T) =& \int_0^\infty \int_0^\infty \ind_{(\frac DT,\infty)}(u) \Big(\frac{T}{F_T}\Big)^{-\alpha} \psi_i \Big(\frac{T}{F_T}h(r,u)\Big) T^{2+\gamma}|g'(Tu)|dr du.\quad i=1,2\label{e:3.19c}.
\end{align}
% Let us also denote $\kappa_1=\alpha$, $\kappa_2=\kappa$.
% By \eqref{e:3.23} and \eqref{e:3.19d} we have
% \begin{equation*}
%  \psi_i(x)\le \abs{x}^{\kappa_i}.
% \end{equation*}
Let us consider $A_1(T)$ first. By \eqref{e:3.23} 
we have
\begin{equation*}
 A_1(T)\le C\int_0^{\frac DT}\int_0^\infty \abs{h(r,u)}^{\alpha} T^{2+\gamma}\abs{g'(Tu)}drdu.
\end{equation*}
Then for $T>1$ by \eqref{h_def}, \eqref{f_def} and Lemma \ref{f_lem} (i) we obtain
\begin{align*}
 A_1(T)\le &C_1 \int_0^{\frac DT}\int_0^{t_n+\frac DT}u^{\alpha} T^{2+\gamma}\abs{g'(Tu)}du\\
=&C_1(t_n+D) T^{1+\gamma-\alpha}\int_0^D u^{\alpha}\abs{g'(u)}du\\
\le& C_1(t_n+D) D^\alpha g(0) T^{1+\gamma-\alpha}\ \to 0.
 \end{align*}
Similarly, using $\kappa<\alpha$,  $(\frac{T}{F_T})^{-\alpha}\le (\frac{T}{F_T})^{-\kappa}$ for $T\ge 1$ and \eqref{e:3.19d} we have
\begin{equation*}
 A_2(T)\le C\int_0^{\frac DT}\int_0^\infty \abs{h(r,u)}^{\kappa} T^{2+\gamma}\abs{g'(Tu)}drdu,
\end{equation*}
and the same argument as above shows that $A_2(T)\to 0$.

Now let us proceed to $B_1(T)$. By \eqref{e:3.25a} and Assumption (\textbf{G}) the integrand in \eqref{e:3.19c} with $i=1$ converges to $C_\psi C_g \abs{h(r,x)}^\alpha u^{-2-\gamma}$. Moreover, by  \eqref{e:3.23} and \eqref{e:3.19a}, it is bounded by
\begin{equation*}
 C_2\abs{h(r,u)}^\alpha u^{-2-\gamma}.
\end{equation*}
By \eqref{h_def} , \eqref{f_def}  part (ii) of Lemma \ref{f_lem} and the fact that $\alpha> 1+\gamma$ the latter function is integrable on $\R_+^2$, hence we can pass to the limit under the integral sign, and \eqref{e:3.21} follows.

It remains to consider $B_2(T)$. 
Using \eqref{e:3.19d} and \eqref{e:3.19a}, $1+\gamma<\kappa<\alpha$, and again Lemma \ref{f_lem} (ii) for $T\ge 1$ we have
\begin{equation*}
 B_2(T)\le C \left(\frac T{F_T}\right)^{\kappa-\alpha}\int_{\R_+^2}\abs{h(r,u)}^\kappa u^{-2-\gamma}drdu\le C \left(\frac T{F_T}\right)^{\kappa-\alpha}\ \to 0.
\end{equation*}
This finishes the proof of \eqref{e:3.22}.
We have proved \eqref{e:fdd_Y_T_i}.

\bigskip
\textbf{Step 2. Tightness.}

Now we continue to establish tightness in $\mathcal{C}([0,\tau])$ for any $\tau>0$.
% We can write $Y_T = Y_{T,1}+Y_{T,2}$, where $Y_{T,1}$ and $Y_{T,2}$ are independent and correspond to L\'{e}vy measures $\nu_1$ and $\nu_2$, respectively.
\begin{comment}
 It is easy to show that there exists a finite constant $K_1$ such that 
\[ |\psi_1(x)|\leq K_1|x|^\alpha, \quad x \in \mathbb{R}.\]
Furthermore, by lemmas \ref{psi2_cond1} and \ref{psi2_cond2}
\[ |\psi_2(x)| \leq K_2|x|^\kappa, \quad x\in \mathbb{R}\]
for some $\kappa>1+\gamma$ and some finite constant $K_2$. Without loss of generality we may also assume that $\kappa \leq \alpha$.
\end{comment}

Let us consider the sequence $(Y_{T,2})$ first. We are going to use Theorem 12.3 in \cite{BILL}. Without loss of generality we may assume that $\alpha > \kappa>1+\gamma$ and $T\ge 1$. Since for each $T\geq 1$ the process $Y_{T,2}$ has stationary increments, one only has to show that there exist $C>0$,  $\beta \geq 0$, $\epsilon>0$  such that
\begin{equation}\label{tightness_condition}
\mathbb{P}(|Y_{T,2}(t)|\geq \lambda)\leq \frac{C}{\lambda^\beta}t^{1+\epsilon}, \quad T\geq 1, t\ge 0, \lambda>0.
\end{equation}
 We will use the following estimate, valid for any real valued random variable $\xi$
\begin{equation} \mathbb{P}(|\xi|>\lambda)\leq\lambda\int_{-2/\lambda}^{2/\lambda}\left(1-\exp(i\theta \xi)\right)d\theta, \quad \lambda>0. 
 \label{e:3.28c}
\end{equation}
% (which holds for scalar random va$\lambda>0$ and $\theta \in \mathbb{R}$)riables $X$, $\lambda>0$ and $\theta \in \mathbb{R}$) 
By  \eqref{Y_T_char_func}, recalling \eqref{f_def} we have
\begin{equation}
\mathbb{E}\exp(i\theta Y_{T,2}(t))=\exp \left(-\int_0^\infty \int_0^\infty T^2\psi_2\left(\frac{\theta T}{F_T}f(t,r,u)\right)|g'(Tu)|drdu\right).
\end{equation}
Hence, using \eqref{e:3.19d}, \eqref{thm1_norming},
the simple inequality $1-e^{-x}\le x$ and the fact that for $T\ge 1$ we have $(T/F_T)^{\kappa}\le (T/F_T)^{\alpha}=T^\gamma$ it follows that 
\begin{align}
 1 - \mathbb{E}\exp(i\theta Y_{T,2}(t)) 
%  &\leq &  1 - \exp \left(-\int_0^\infty \int_0^\infty T^2 K_2\left|\frac{\theta T}{F_T}f(t,r,u)\right|^\kappa|g'(Tu)|drdu\right) \nonumber\\
 &\leq  C\int_0^\infty \int_0^\infty T^2 \Big|\frac{\theta T}{F_T}f(t,r,u)\Big|^\kappa|g'(Tu)|drdu \notag\\
& = C  |\theta|^\kappa\int_0^\infty \int_0^\infty\big|f(t,r,u)\big|^\kappa T^{2+\gamma}|g'(Tu)|drdu \nonumber \\
 &= C|\theta|^\kappa\big(J_1(T) + J_2(T)\big)\label{pthm1_loc1},
\end{align}
where
\[ J_1(T) = \int_0^1 \int_0^\infty \big|f(t,r,u)\big|^\kappa T^{2+\gamma}|g'(Tu)|drdu\]
and
\[ J_2(T) = \int_1^\infty \int_0^\infty \big|f(t,r,u)\big|^\kappa T^{2+\gamma}|g'(Tu)|drdu.\]
Notice that for $u\in (1,\infty)$ and all $T$ sufficiently large $T^{2+\gamma}|g'(Tu)| \leq C |u|^{-2-\gamma}$ for some finite positive constant $C$. Thus, by Lemma \ref{f_lem} (ii) we have
\begin{equation}
\label{e:3.28a}J_2(T) \leq C_1 t^{\kappa-\gamma} 
\end{equation}
for all $T$ large and some finite constant $C_5$. Now, let $\epsilon>0$ be such that $\kappa>1+\gamma+\epsilon$. By \eqref{f_def} and  \eqref{f_estimates}, and then using $\int_0^\infty \ind_{[r-u,r]}(s)dr=u$ for $s,r>0$, we see that
\begin{align*}
J_1(T) &\leq 
% \int_0^1 \int_0^\infty f(t,r,u)t^\epsilon u^{\kappa-1-\epsilon}T^{2+\gamma}|g'(Tu)|drdu \\
% & = & 
\int_0^1 \int_0^\infty  \Big(\int_0^t \ind_{[r-u, r]}(s)ds\Big)t^\epsilon u^{\kappa-1-\epsilon}T^{2+\gamma}|g'(Tu)|drdu \\
 &=  t^\epsilon\int_0^1 tu u^{\kappa-1-\epsilon}T^{2+\gamma}|g'(Tu)|du \\
%  &=  t^{1+\epsilon}\int_0^1  u^{\kappa-\epsilon}T^{2+\gamma}|g'(Tu)|du \\
& =  t^{1+\epsilon} T^{1+\gamma+\epsilon - \kappa}\int_0^T   u^{\kappa-\epsilon}|g'(u)|du\\
& \le t^{1+\epsilon}\int_0^\infty u^{\kappa -\epsilon}\abs{g'(u)}du.
\end{align*}
Let $D$ be as in \eqref{e:3.19a}, then
\begin{equation}
 J_1(T)\le t^{1+\epsilon}\left(D^{\kappa-\epsilon}\int_0^D\abs{g'(u)}du+2C_g\int_D^\infty u^{\kappa -2-\gamma-\epsilon}du\right)\le C t^{1+\epsilon},
\label{e:3.28b}
 \end{equation}
since the first integral is bounded by $g(0)$, and the second is finite thanks to the choice of $\epsilon$.
Combining \eqref{e:3.28b}, \eqref{e:3.28a}, \eqref{pthm1_loc1} with \eqref{e:3.28c} yields \eqref{tightness_condition} (here $\beta=\kappa$).
% By assumption for large $u$ the function under the integral behaves as $u^{\kappa - \epsilon - 2 -\gamma}$, while for small $u$ it is bounded by $\abs{g'(u)}$
% there is a constnat $C_6$ such that the function $u \mapsto u^{\kappa - \epsilon}|g'(u)|$ is  bounded by $C_6 u^{\kappa - \epsilon - 2 -\gamma}$ for all $u$ sufficiently large. This implies that 
% \[\int_0^T   u^{\kappa-\epsilon}|g'(u)|du \leq C_7 T^{\kappa - \epsilon - 1 - \gamma}\]
% for some finite $C_7$ and all $T$ sufficiently large. 
% Hence
% \[ J_1(T) \leq C_7 t^{1+\epsilon} \]
for all $t\geq 0$ and all $T$ large enough.
This finishes the proof of tightness of 
% \eqref{tightness_condition} holds and 
$Y_{T,2}$ in $\mathcal{C}([0, \tau])$ % for all $\tau>0$.

The proof of tightness $Y_{T,1}$ is similar. We have an analogue of \eqref{pthm1_loc1} with $\alpha$ instead of $\kappa$ and the same argument works. In this case $\epsilon=\alpha-1-\gamma$.

Combined with convergence of finite dimensional distributions this implies convergence of $Y_T$ in $C([0,\tau])$ for any $\tau>0$.
\qed
% \end{proof}
% 
% \begin{rem}
% It is easy to show that tightness will also hold for more general function $g$, e.g., we may assume that $g$ has a bounded derivative $g'$ which is eventually negative and $g(\cdot)$ is regularly varying at infinite with exponent $\gamma$. 
% \end{rem}
% 

\subsection{Proof of Theorem \ref{THM2}}

% \begin{proof}

We will show convergence of finite-dimensional distributions by proving the convergence their characteristic functions.

According to the general scheme, we  fix any  $a_1,\ldots, a_n\in \R$,  $0\le t_1\ldots \le t_n$ and we start with formula \eqref{Y_T_char_func}.  To prove the theorem it suffices to show that for 
$I(T)$ defined by \eqref{e:I_T} and \eqref{h_T_def}
we have
\begin{equation}
 \lim_{T\to \infty}I(T)=K^{1+\gamma}\int_0^\infty \abs{a(r)}^{1+\gamma} dr,
\label{e:3.36a}
 \end{equation}
where
\begin{equation}
 a(r)=\sum_{j=1}^n a_j\ind_{[0,t_j]}(r)
 \label{e:a}
\end{equation}

% Take any $a_1,\ldots, a_n \in \mathbb{R}$ and $t_1,\ldots, t_n \geq 0$. 
% Recall that
% \begin{equation}\label{fdd_conv_1+gamma2}
% \log \mathbb{E}\exp\Big(i\sum_{j=1}^n a_j  Y_T(t_j)\Big) = - \int_{\mathbb{R}^2}\psi\left(\frac{1}{F_T}h_T(r,x)\right)drdx 
% \end{equation}
% where $h_T$ is as in \eqref{h_T_def}. Notice that for $u,t\geq 0$ we have
% \begin{equation}\label{f_rep}
% r_+ \wedge t - (r-u)_+\wedge t = 
% \int_0^\infty \ind_{[0,t]}(s)\ind_{[r-u,r]}(s)ds.
% \end{equation}
Recalling the definition of $h_T$ (see \eqref{h_T_def}) and substituting $r'=\frac rT$, $u'=\frac u{F_T}$ and then $s'=\frac {(s-r)} u\frac T {F_T}$  we obtain
% Changing the variables in the integral on the right-hand side of \eqref{fdd_conv_1+gamma2} we may rewrite it as
% \begin{eqnarray}\label{loc4_1}
% \int_0^\infty \int_0^\infty TF_T\psi\Big(u\frac{T}{F_Tu} \sum_{j=1}^n a_j \int_0^\infty \ind_{[0, t_j]}(s)\ind_{[r-\frac{uF_T}{T}, r]}(s)ds\Big)|g'(F_Tu)|drdu.
% \end{eqnarray}
\begin{align}
I(T)=&
\int_0^\infty \int_0^\infty TF_T\psi\Big(\frac{T}{F_T} \int_0^\infty a(s)\ind_{[r-\frac{uF_T}{T}, r]}(s)ds\Big)|g'(F_Tu)|drdu\notag
\\
=&\int_0^\infty \int_0^\infty \psi\Big(u \int_{-1}^0 a(r+\frac u T F_Ts)ds\Big)F_T^{2+\gamma }|g'(F_Tu)|drdu,
\label{loc4_1}
\end{align}
where in the last equality we also used $T=F_T^{1+\gamma}$.

By Assumption (\textbf{G}) it is now clear that the integrand in \eqref{loc4_1} converges pointwise to $C_g\psi(ua(r))u^{-2-\gamma}$. Also notice, that making the substitution $u'=ua(r)$ we have
\begin{equation}
 \int_0^\infty \int_0^\infty C_g\psi(ua(r))u^{-2-\gamma}dudr
 =C_g\int_0^\infty \psi(u)u^{-2-\gamma}du\int_0^\infty \abs{a(r)}^{1+\gamma}dr.
 \label{e:3.42}
\end{equation}
The integral with respect to $u$ on the right hand side of \eqref{e:3.42} is finite by \eqref{thm2_cond1} or \eqref{thm2_cond2a}, hence \eqref{e:3.36a} will follow provided we can justify passing to the limit under the integrals.

% Notice that the quantity
% \[\frac{T}{F_Tu} \sum_{j=1}^n a_j \int_0^\infty \ind_{[0, t_j]}(s)\ind_{[r-\frac{uF_T}{T}, r]}(s)ds\]
% is bounded uniformly in $T,u$ and $r$ by $\sum_{j=1}^n|a_j|$ and converges pointwise, as $T\rightarrow \infty$ to $a(r):= \sum_{j=1}^n a_j \ind_{[0,t_j]}(r)$ for any $u,r>0$. Furthermore, the quantity
% \[ |g'(F_Tu)|TF_T = |g'(F_Tu)||uF_T|^{2+\gamma}|u|^{-2-\gamma}\]
% converges to $C_g|u|^{-2-\gamma}$ as $T\rightarrow\infty$. Thus, for any $u,r>0$, the integrand in \eqref{loc4_1} converges to 
% \[C_g\psi(u a(r))|u|^{-2-\gamma}.\]
% Moreover, 
% \[\int_0^\infty \int_0^\infty\psi(u a(r))|u|^{-2-\gamma} dr du <\infty.\]
% Indeed, 
% \begin{eqnarray}
% \int_0^\infty \int_0^\infty\psi(u a(r))|u|^{-2-\gamma} dr u &=& \int_0^\infty |a(r)|^{1+\gamma}dr\int_0^\infty \psi(u)|u|^{-2-\gamma}du, \nonumber
% \end{eqnarray}
% with the second integrand on the right-hand side of the above equation being finite by assumption (see Remark \ref{thm2_REM1}) and  $\int_0^\infty |a(r)|^{1+\gamma}dr < \infty$ because $a$ is supported on $[0, t_{max}]$, where $t_{max}:=\max(t_1,\ldots, t_n)]$. Therefore we must only justify the passage to the limit under the integral in \eqref{loc4_1}. 
Now the proof forks into two parts depending on whether we assume (i) or (ii) in the formulation of Theorem \ref{THM2}. 

Consider first the case when (i) is satisfied. Using Assumption (\textbf{G}) choose $D>0$ such that \eqref{e:3.19a} holds.
%$|g'(z)|\leq 2C_g|z|^{-2-\gamma}$ for $z>D$.
Suppose  that $T$ is such that $T>1$ and $T>D$. Observing that since the support of $a$ is $[0,t_n]$ and hence the integrand in \eqref{loc4_1} is equal to zero if $r>t_n+u\frac {F_T}T$ we write
\begin{equation}
 I(T)=I_1(T)+I_2(T)+I_3(T),
 \label{e:3.43}
\end{equation}
where 
\begin{align*}
 I_1(T)=&\int_0^{\frac D{F_T}} \int_0^{t_n+u\frac {F_T}{T}}\psi\Big(u \int_{-1}^0 a(r+\frac u T F_Ts)ds\Big)F_T^{2+\gamma }|g'(F_Tu)|drdu,\\
 I_2(T)=&\int_{\frac D{F_T}}^{\frac T{F_T}} \int_0^{t_n+u\frac {F_T}{T}} \psi\Big(u \int_{-1}^0 a(r+\frac u T F_Ts)ds\Big)F_T^{2+\gamma }|g'(F_Tu)|drdu,\\
I_3(T)=&\int_{\frac T{F_T}}^\infty\int_0^{t_n+u\frac {F_T}{T}} \psi\Big(u \int_{-1}^0 a(r+\frac u T F_Ts)ds\Big)F_T^{2+\gamma }|g'(F_Tu)|drdu\\
\end{align*}
By
% Using the fact that the suuport of $a$ is $[0,t_n]$ and 
\eqref{thm2_cond1} we have
% The integral \eqref{loc4_1} may be written as $A_1(T)+A_2(T)+A_3(T)$, according to whether $u\in(0,K_1/F_T)$, $u \in (K_1/F_T, T/F_T)$ or $u \in (T/F_T,\infty)$, respectively. For all $T$ sufficiently large we have
\begin{align}
I_1(T) 
% &\le &C \int_0^{D/F_T} \int_0^{t_n+u\frac {F_T}t} F_T^{2+\gamma}\psi\Big(u\frac{T}{F_Tu} \sum_{j=1}^n a_j \int_0^\infty \ind_{[0, t_j]}(s)\ind_{[r-\frac{uF_T}{T}, r]}(s)ds\Big)\nonumber \\
% && \: |g'(F_Tu)|drdu \nonumber \\
&\leq  C\int_0^{D/F_T} \int_0^{t_{n}+\frac{uF_T}{T}}\Big|u \norm{a}_\infty\Big|^\kappa F_T^{2+\gamma} |g'(F_Tu)|drdu \notag \\
& \leq  C_1 (t_n+D) \int_0^{D/F_T}u^{\kappa} F_T^{2+\gamma}|g'(F_Tu)|drdu \notag \\
= & C_1 (t_n+D)F_T^{1+\gamma-\kappa}\int_0^{D} u^{\kappa}\abs{g'(u)}du\notag\\
 \leq & C_1 (t_n+D)D^{\kappa} g(0) {F_T^{1+\gamma -\kappa}} \to 0,
 \label{loc4_3}
\end{align}
since we have assumed that $\kappa>1+\gamma$.

% 
% 
% 
% The quantity in \eqref{loc4_3} vanishes as $T\rightarrow \infty$ since $\kappa>1+\gamma$ and $g'$ is integrable. 

Now we consider $I_2(T)$. The integrand converges pointwise to  $\psi(ua(r))u^{-2-\gamma}$. Moreover, by assumption \eqref{thm2_cond1} and the fact that the support of $a$ is $[0,t_n]$, for $ D/{F_T}\le u\le T/F_T$ we have
\begin{equation}
 \psi\Big(u \int_{-1}^0 a(r+\frac u T F_Ts)ds\Big)F_T^{2+\gamma }|g'(F_Tu)|
 \le C\ind_{[0,t_n+1]}(r) (u^\kappa \wedge u^\alpha)u^{-2-\gamma}.
\label{e:3.44a}
 \end{equation}
The latter function is integrable on $\R_+^2$. Hence, using also \eqref{e:3.42} we see that
\begin{equation}
 \lim_{T\to \infty}I_2(T)=K^{1+\gamma}\int_0^\infty \abs{a(r)}^{1+\gamma} dr.
\label{e:3.45}
 \end{equation}
 
% The integrand in $A_2(T)$ may, for all $T$ sufficiently large, be bounded by the function 
% \begin{eqnarray}
% (r,u) &\mapsto & \ind_{[0, t_{max}+\frac{uF_T}{T}]}(r)TF_TC_2|u|^\kappa \wedge|u|^\alpha |g'(F_Tu)| \nonumber \\
% & \leq & \ind_{[0, t_{max}+\frac{uF_T}{T}]}(r)TF_TC_2|u|^\kappa \wedge|u|^\alpha 2C_g|F_Tu|^{-2-\gamma} \nonumber\\
% & \leq & C_4 \ind_{[0, t_{max}+u]}(r)|u|^\kappa \wedge|u|^\alpha |u|^{-2-\gamma} \nonumber \\
% & \leq & C_4 \ind_{[0, t_{max}+1]}(r)|u|^\kappa \wedge|u|^\alpha |u|^{-2-\gamma}, \label{loc4_5}
% \end{eqnarray}
% where the inequality in \eqref{loc4_5} follows from the fact that we only consider $u<T/F_T$ in $A_2(T)$. It is easy to check that the function 
% \[ u \mapsto |u|^\kappa \wedge|u|^\alpha |u|^{-2-\gamma},\]
% is integrable on $(0, \infty)$. 
Now we proceed to $I_3(T)$. Observe that since $\abs{a(s)}\le \norm{a}_\infty\ind_{[0,t_n]}(s)$ we have
\begin{equation}
\abs{ u \int_{-1}^0 a(r+\frac u T F_Ts)ds} \le \int_\R\abs{a(s)}ds\frac T{F_T} \le \norm{a}_\infty t_n\frac T{F_T}.
 \label{e:3.46a}
\end{equation}
Thus, using
\eqref{thm2_cond1} we can estimate
\begin{align}
I_3(T) & \leq  C\int_{T/F_T}^\infty (t_n+u\frac {F_T}T) \left(\frac {T}{F_T}\right)^\alpha u^{-2-\gamma} du \notag\\
& = \left(\frac T{F_T}\right)^{\alpha-1-\gamma}\int_1^\infty (t_n+u)u^{-2-\gamma}du\to 0.\label{e:3.46}
\end{align}
by asumption $\alpha<1+\gamma$ and the form of $F_T$.
% It is simple algebra to check that for all $T$ sufficiently large, \eqref{loc4_4} is bounded by 
% \[C_8(t_{max}+1)T^{\alpha-1-\gamma}F_T^{\alpha-1},\]
% which vanishes as $T\rightarrow\infty$ since we have assumed $\gamma>\alpha-1$ and $\gamma\in (0,1)$. 

From \eqref{e:3.43}, \eqref{loc4_3},  \eqref{e:3.45} and \eqref{e:3.46} we obtain 
\eqref{e:3.36a} in case (i) which completes the proof of convergence of finite dimensional distributions in this case. 

\medskip

Now consider the case  (ii) in the formulation of Theorem \ref{THM2} is satisfied.  We again have \eqref{loc4_1} and \eqref{e:3.43}.
Now for $I_1+I_2$ we can proceed in a similar way as for $I_2$ in case (i). The only difference is that instead of  \eqref{e:3.44a}  for $0\le u\le \frac T{F_T}$, we use 

\begin{equation*}
 \psi\Big(u \int_{-1}^0 a(r+\frac u T F_Ts)ds\Big)F_T^{2+\gamma }|g'(F_Tu)|
 \le C\ind_{[0,t_n+1]}(r) \psi (\norm a_\infty u)u^{-2-\gamma},
 \end{equation*}
 since we now assume that $\psi$ is nondecreasing on $\R_+$. Similarly as above we obtain that $I_1(T)+I_2(T)$ converge, as $T\to \infty$, to the right hand side of \eqref{e:3.36a}
 
 For $I_3$ we again use \eqref{e:3.46a} and monotonicity of $\psi$ on $\R_+$ obtaining
 \begin{align*}
  I_3(T)& \le  C\int_{T/F_T}^\infty (t_n+u\frac {F_T}T) \psi\left(\frac {T}{F_T}\norm a_\infty t_n\right) u^{-2-\gamma} du  \notag\\
 &= \left(\frac T{F_T}\right)^{-1-\gamma}\psi(\frac T{F_T}\norm a_\infty t_n)\int_1^\infty (t_n+u)u^{-2-\gamma}du.
 \end{align*}
It now suffices to notice that $T^{-1-\gamma}\psi(T)$ converges to $0$ as $T\to \infty$, since by the fact that $\psi$ is nondecreasing
\begin{equation*}
 \frac 1{1+\gamma }\psi(T)T^{-1-\gamma}= \int_T^\infty \psi(T) x^{-2-\gamma}dx\le  \int_T^\infty \psi(x)x^{-2-\gamma}dx.
\end{equation*}
The last integral converges to $0$ by \eqref{thm2_cond2a}.
This proves that $I_3(T)$ converges to $0$. The proof in case (ii) is complete. \qed

\subsection{Proof of Theorem \ref{THM3}}

% \begin{proof}

We use the decomposition \eqref{e:3.5a}. Using the estimate $1-\cos(\theta x)\le (\theta x)^2$, \eqref{e:3.5b} and \eqref{e:nu} we have  $\psi_1(x)\le C {x^2}$. This together with the assumption \eqref{e:2.18a} implies 
\begin{equation*}
 \psi_1(x)\le C\abs{x}^2\wedge \abs{x}^\kappa.
\end{equation*}
The assumptions of   Theorem  \ref{THM2}, in which we take  $\bar \alpha=\kappa$ and $\bar \kappa =2$, are satisfied for $\psi_1$ and the process
%\begin{equation*}
$( {F_T}/{T^{1+\gamma}})Y_{T,1}$
%\end{equation*}
converges in the sense of finite dimensional distributions. $F_T=T^\frac 1\alpha$ with $\alpha<1+\gamma$ hence the above implies that
\begin{equation*}
 Y_{T,1}\fdd 0.
\end{equation*}
And therefore also $Y_{T,1}(t)$ converges to $0$ in probability for any $t\ge 0$.

From now on we may therefore assume that
% $\nu$
% 
% Notice first that can write $Y_T = Y_{T,1}+Y_{T,2}$, where $Y_{T,1}$ and $Y_{T,2}$ are independent and correspond to L\'{e}vy measures $\nu_1$ and $\nu_2$, respectively. Since $\kappa>\alpha$ and $\psi_1$ is bounded, the process $Y_{T_1}$ satisfies the assumptions of Theorem \ref{THM2} and converges in the sense of finite-dimensional distributions with the norming $T^{\frac{1}{1+\gamma}}$ and since we assume in Theorem \ref{THM3} that $1+\gamma>\alpha$, the process $Y_{T,1}$ converges to zero with the norming $F_T = T^{1/\alpha}$. Therefore form now on we may assume that 
$\nu(\{|x|\leq 1\})=0$ and $\psi=\psi_2$. In what follows we omit the index $2$.
Observe that in this case $\psi$ is bounded since $\nu$ is finite (cf. {\eqref{e:nu}}) and from assumption \eqref{e:2.16a} it follows that 
\begin{equation}
 \psi(x)\le C(\abs{x}^\alpha\wedge 1).
 \label{e:3.50}
\end{equation}

Take any $a_1,\ldots,a_n \in \mathbb{R}$ and $0\le t_1\le \ldots\le t_n$ in $\mathbb{R}_+$. According to the general scheme (cf. \eqref{Y_T_char_func}) we need to show that for $I(T)$  given by \eqref{e:I_T}  and $a$ by \eqref{e:a} we have 
\begin{equation}
 \lim_{T\to \infty} I(T)=K^{\alpha}\int_0^\infty \abs{a(r)}^{\alpha}dr.
 \label{e:3.51}
\end{equation}

Using \eqref{e:I_T}, \eqref{h_T_def}, \eqref{f_def}, \eqref{e:3.11a} and substituting $r'=\frac{r-u}{T}$ we rewrite $I(T)$ as
\begin{equation}
 I(T)=I_1(T)+I_2(T),
 \label{e:3.52}
\end{equation}
where
\begin{align}
I_1(T) &=\int_0^\infty \int_{-u/T}^{0}T\psi\Bigg(\frac{1}{F_T}\sum_{j=1}^n a_j f(Tt_j,Tr+u,u)\Bigg)|g'(u)|drdu, \label{integrandDCT}\\
I_2(T) &= \int_0^\infty \int_{0}^{t_n}T\psi\Bigg(\frac{1}{F_T}\sum_{j=1}^n a_j f(Tt_j,Tr+u,u)\Bigg)|g'(u)|drdu. \label{integrandDCT2}
\end{align}
Observe that, by \eqref{f_def},
for $u\ge 0$ and $r\ge 0$ we have
\begin{equation}
 \lim_{T\to \infty} f(Tt_j,Tr+u,u)=\lim_{T\to \infty}\int_0^{Tt_j}\ind_{[Tr,Tr+u]}(s)ds=u\ind_{[0,t_j)}(r),
 \label{e:3.55}
\end{equation}
and
\begin{equation}
 f(Tt_j,Tr+u,u)\le u.
 \label{e:3.56}
\end{equation}

Using \eqref{integrandDCT}, \eqref{e:3.50} and  \eqref{e:3.56} we have 
% and \eqref{thm3_norming} we have
\begin{equation}
 I_1(T)\le C\int_0^\infty u\left(\left(\frac {u\sum_{j=1}^n\abs{a_j}}{F_T}\right)^\alpha \wedge 1\right) \abs{g'(u)}du\longrightarrow 0,
 \label{e:3.57}
\end{equation}
since the function under the integral converges pointwise to $0$ and is bounded by $u\abs{g'(u)}$, which is integrable by Assumption ({\bf{G}}).

Now we proceed to $I_2(T)$. By \eqref{integrandDCT2}, \eqref{e:3.55}, \eqref{e:2.16a} and \eqref{thm3_norming} we see that
\begin{equation*}
 \lim_{T\to \infty}T\psi\Bigg(\frac{1}{F_T}\sum_{j=1}^n a_j f(Tt_j,Tr+u,u)\Bigg)|g'(u)|=\abs{a(r)}^\alpha \abs{g'(u)}\qquad a.e.
\end{equation*}
and by \eqref{e:3.50}
\begin{equation*}
 T\psi\Bigg(\frac{1}{F_T}\sum_{j=1}^n a_j f(Tt_j,Tr+u,u)\Bigg)|g'(u)| \le C u^\alpha \abs{g'(u)},
\end{equation*}
The function on the right hand side is integrable on $\R_+\times[0,t_n]$. Hence
\begin{equation}\label{e:3.58}
 \lim_{T\to \infty}I_2(T)=\int_0^\infty\int_0^{t_n}\abs{a(r)}^\alpha \abs{g'(u)}drdu
 =g(0)\int_0^\infty \abs{a(r)}^\alpha dr.
\end{equation}
From \eqref{e:3.52}, \eqref{e:3.57} and \eqref{e:3.58} we obtain \eqref{e:3.51}, thus finishing the proof of the theorem. \qed

\subsection{Proof of Theorem \ref{THM_CRITICAL}}

% \begin{proof}
Take any $a_1,\ldots, a_n \in \mathbb{R}$ and $0\le t_1\le \ldots\le  t_n \geq 0$. Recall the general formula for the characteristic function of finite dimensional distributions of $Y_T$ \eqref{Y_T_char_func} and the notation \eqref{h_T_def}, \eqref{f_def}, \eqref{e:I_T} and \eqref{e:a}. 
By Lemma \ref{Y_T_char_lemma}, to prove the desired convergence of finite dimensional distributions it suffices to show
\begin{equation}
 \lim_{T\to \infty }I(T)=K^{\alpha}\int_0^\infty \abs{a(r)}^{\alpha}dr.
 \label{e:3.68}
\end{equation}

% Recall that
% \begin{equation}\label{thm4_loc1}
% \log \mathbb{E}\exp\Big(i\sum_{j=1}^n a_j  Y_T(t_j)\Big) = - \int_{\mathbb{R}^2}\psi\left(\frac{1}{F_T}h_T(r,x)\right)drdx 
% \end{equation}
% where $h_T$ is as in \eqref{h_T_def}. 
Since $\psi(x)=|x|^\alpha$, using \eqref{h_T_def}, \eqref{f_def}  and then substituting $r'=\frac rT$ and $s'=\frac sT$ we may 
 rewrite $I_T$ as
\begin{align}
I(T) &= 
\int_0^\infty \int_0^\infty F_T^{-\alpha}\bigg|\sum_{j=1}^n a_j \int_0^\infty
\ind_{[r-u,r]}(s)\ind_{[0,Tt_j]}(s)ds
\bigg|^\alpha|g'(u)|drdu \notag \\
&= 
\int_0^\infty \int_0^\infty TF_T^{-\alpha}\bigg|T\int_0^\infty a(s)\ind_{[r-\frac{u}{T}, r]}(s)ds\bigg|^\alpha|g'(u)|drdu .\notag \\
\intertext{Now we use the form of $F_T$ and of $g$, then make a change of variables $s'=\frac {(r-s)}u T$, and then, finally, substitute $u'=u/\log T$, obtaining}
I(T)& =  \frac{(1+\gamma)}{\log T} \int_0^\infty \int_0^{\infty} u^\alpha\bigg|\int_0^1 a(r-su/T)ds\bigg|^\alpha (1+u)^{-2-\gamma }drdu \notag \\
% & = &  \frac{C_1}{\log T} \int_0^\infty \int_0^\infty u^\alpha\bigg|\int_0^\infty a(r-su/T)\ind_{[0,1]}(s)ds\bigg|^\alpha \frac{1}{(1+u)^{2+\gamma}}drdu \nonumber \\
& = (1+\gamma) \int_0^\infty \int_0^{\infty}\frac{(u\log T)^\alpha}{(1+u\log T)^{1+\alpha}} \bigg|\int_0^1 a(r-su\log T/T)ds\bigg|^\alpha drdu. \label{thm4_loc2}
\end{align}
Now we write 
\begin{equation}
 I(T)=(1+\gamma)(I_1(T) + I_2(T) + I_3(T)),\label{e:3.70}
\end{equation}
where
\begin{align*}
 I_1(T)& =\int_0^1\int_0^{\infty}\ldots \  drdu,\\
 I_2(T)& =\int_1^{T/\log T}\int_0^{\infty}\ldots \  drdu,\\
 I_3(T)& =\int_{T/\log T}^\infty\int_0^{\infty}\ldots \  drdu,\\
\end{align*}
where $\ldots $ stands for the function under the integral in \eqref{thm4_loc2}.
% 
% split the integral as $I(T)=C_1(I_1(T) + I_2(T) + I_3(T))$ with
Let us consider first $I_2(T)$. We make a change of variables $u'=\frac {\log u}{\log T}$ obtaining
\begin{multline}
I_2(T)\\
% & = &\frac{1}{\log T}\int_1^{T/\log T}\int_0^\infty \frac{(u\log T)^{1+\alpha}}{(1+u\log T)^{1+\alpha}}\frac{1}{u} \nonumber \\
% && \: \bigg|\int_0^\infty a(r-su\log T/T)\ind_{[0,1]}(s)ds\bigg|^\alpha dr  du \\
 =  \int_0^{1-\log\log T /\log T} \int_0^\infty\bigg( \frac{T^u\log T}{1+T^u\log T}\bigg)^{1+\alpha} 
 \bigg|\int_0^1 a(r-sT^u\log T/T)ds\bigg|^\alpha dr du,
\end{multline}
% where the second equality is due to the change of variables. 
Notice that 
$\log\log T/\log T$ goes to zero as $T\rightarrow\infty$. Moreover, we have pointwise convergence to $\abs{a(r)}^\alpha$.
% \[ \lim_{T\rightarrow \infty}\bigg|\int_0^\infty a(r-sT^u\log T/T)\ind_{[0,1]}(s)ds\bigg|^\alpha = |a(r)|^\alpha.\]
We have
 $|a(r)|\leq C \ind_{[0,t_n]}(r)$ for some finite constant $C$ hence the upper limit in the integral with respect to $r$ can be replaced by $t_n+1$, since for $r>t_n+1$ the function under the integral with respect to $drdu$ vanishes.
We may use the dominated convergence theorem obtaining
\begin{equation} \lim_{T\rightarrow \infty} I_2(T) = \int_0^\infty |a(r)|^\alpha dr.
\label{e:3.71}
\end{equation}

Let us consider $I_1(T)$ next. We have
\begin{align}
I_1(T)& = \int_0^1 \int_0^{t_n+1} \frac{(u\log T)^{\alpha}}{(1+u\log T)^{1+\alpha}
 }
% \notag \\
% &\hskip 1cm \:
\bigg|\int_0^1 a(r-su\log T/T) ds\bigg|^\alpha dr  du \notag\\
& \leq \norm a_\infty (1+t_n)\int_0^1 \frac{(u\log T)^{\alpha}}{(1+u\log T)^{1+\alpha}} du \notag \\
& =  C \int_0^{\log T} \frac{u^\alpha}{(1+u)^{\alpha+1}}\frac{1}{\log T}du \notag\\
& \leq C_1 \Big(\frac{1}{\log T}\int_0^1 \frac{u^\alpha}{(1+u)^{\alpha+1}} + 
\frac{1}{\log T} \int_1^{\log T} \frac{1}{u}du \Big) \notag\\
& \leq C_2 \Big(\frac{1}{\log T} + \frac{\log\log T}{\log T}\Big)\to 0,
\label{e:3.72}
\end{align}
% and the last quantity vanishes as $T\rightarrow \infty$. 
It remains to show that $I_3(T)$ also converges to $0$ as $T\to \infty$. 
Taking into account that the support of $a$ is $[0,t_n]$, after a change of variables we have
% \[ I_3(T) = \frac{1}{\log T} \int_T^\infty \int_0^{t_n+u/(T\log T)}  \frac{u^\alpha}{(1+u)^\alpha} \Big|\int_0^{1} a(r-su/T) (s)ds\Big|^\alpha dr  du.\]
% On the set $\{(s,u,r): u \geq T, 0 \leq s \leq 1, r\geq (t_{max}+1)\frac{u}{T}\}$ we have
% \[ r - \frac{su}{T} \geq t_{max}\frac{u}{T} + (1-s)\frac{u}{T} \geq t_{max}\frac{u}{T} \geq t_max,\]
% hence
% \[|a(r-su/T)| = 0\]
% on this set. Therefore, 
\begin{eqnarray}
I_3(T) 
&=& \frac{1}{\log T}\int_T^\infty \bigg( \int_0^{t_{n}+u/T} \frac{u^\alpha}{(1+u)^{1+\alpha}}\Big|\int_0^\infty a(r-su/T)\ind_{[0,1]}(s)ds\Big|^\alpha dr \bigg)du \nonumber \\
& \leq &   \frac{1}{\log T}\int_T^\infty \bigg( \int_0^{(t_{n}+1) u/T}\frac{T^\alpha}{(1+u)^{1+\alpha}}\Big|\frac{u}{T}\int_0^\infty \ind_{[0, t_{n}]}(r - su/T)ds\Big|^\alpha dr \bigg) du \nonumber \\
& \leq & C\frac{(1+t_n)}{\log T} \int_T^\infty \frac{u}{T}\frac{T^\alpha}{u^{\alpha+1}}du \nonumber \\
& = &  \frac{C_1}{\log T} \int_T^\infty \frac{T^{\alpha-1}}{u^\alpha}du \nonumber\\
&=& C_2 \frac{1}{\log T}\to 0
\label{e:3.73}
\end{eqnarray}
Combining \eqref{e:3.70}-\eqref{e:3.73} shows that  \eqref{e:3.68} is satisfied. This
finishes the proof of the theorem.
\qed
% \end{proof}

% 
% 
% 
% \begin{appendices}
% \section{TBN}
% a
% \end{appendices}

\bibliographystyle{plain}

\bibliography{publications}

\end{document}